\documentclass[a4paper,10pt]{article}
\usepackage{amsfonts}
\usepackage{amstext}
\usepackage{amsthm}
\usepackage{amsmath}
\usepackage{amssymb}
\usepackage{latexsym}
\usepackage{graphicx}
\usepackage[latin1]{inputenc}
\newcommand{\R}{\Bbb{R}}
\newcommand{\N}{\Bbb{N}}
\newcommand{\Z}{\Bbb{Z}}
\newcommand{\s}{\Bbb{S}}

\newtheorem{teor}{Theorem}[section]
\newtheorem{propo}{Proposition}[section]

\newtheorem{cor}{Corollary}[section]

\newcommand{\n}{\noindent}

\begin{document}

\title{Minimization to the Zhang's energy on {\Large $BV(\Omega)$} and\\ sharp affine Poincaré-Sobolev inequalities
\footnote{Key words: Affine energy, affine Sobolev inequality, compactness of affine immersion, constrained minimization}
}

\author{\textbf{Edir Junior Ferreira Leite \footnote{\textit{E-mail addresses}:
edirjrleite@ufv.br (E.J.F. Leite)}}\\ {\small\it Departamento de Matem\'{a}tica,
Universidade Federal de Vi\c{c}osa,}\\ {\small\it CCE, 36570-900, Vi\c{c}osa, MG, Brazil}\\
\textbf{Marcos Montenegro \footnote{\textit{E-mail addresses}:
montene@mat.ufmg.br (M.
Montenegro)}}\\ {\small\it Departamento de Matem\'{a}tica,
Universidade Federal de Minas Gerais,}\\ {\small\it Caixa Postal
702, 30123-970, Belo Horizonte, MG, Brazil}}

\date{}{

\maketitle

\markboth{abstract}{abstract}
\addcontentsline{toc}{chapter}{abstract}

\hrule \vspace{0,2cm}

\n {\bf Abstract}

We prove the existence of minimizers for some constrained variational problems on $BV(\Omega)$, under subcritical and critical restrictions, involving the affine energy introduced by Zhang in \cite{Z}. Related functionals have non-coercive geometry and properties like lower semicontinuity and affine compactness are deeper in the weak* topology. As a by-product of our developments, extremal functions are shown to exist for various affine Poincaré-Sobolev type inequalities.

\vspace{0.5cm}
\hrule\vspace{0.2cm}

\section{An overview and the first statements}

Variational problems have been quite studied in the space of functions of bounded variation $BV(\Omega)$, mostly in connection with existence of solutions for equations in the presence of the $1$-Laplace operator, as for instance in the famous Cheeger's problem \cite{Ch}. Contributions along this line of research are given in the works \cite{BW, BS, DM, D, KS, LL, LS}, among others. Part of them focus more specifically on the problem of minimizing the functional
\[
\Phi(u) = |Du|(\Omega) + \int_\Omega a |u|\, dx + \int_{\partial \Omega} b |\tilde{u}|\, d{\cal H}^{n-1}
\]
on the entire $BV(\Omega)$ space or constrained to some subset of it, where $\Omega$ denotes a bounded open in $\R^n$ with Lipschitz boundary, $n \geq 2$, $a \in L^\infty(\Omega)$ and $b \in L^\infty(\partial \Omega)$. Here, $|Du|(\Omega)$, $\tilde{u}$ and ${\cal H}^{n-1}$ stand respectively for the total variation measure of $u$ on $\Omega$, the trace of $u$ on $\partial \Omega$ and the $(n-1)$-dimensional Hausdorff measure on $\partial \Omega$.

Two subsets of $BV(\Omega)$ typically considered are:
\begin{eqnarray*}
&& X = \{ u \in BV(\Omega):\ \int_\Omega |u|^q\, dx = 1\},\\
&& Y = \{ u \in BV(\Omega):\ u \in X,\ \int_\Omega |u|^{r-1} u\, dx = 0\}
\end{eqnarray*}
for exponents $1 \leq q, r \leq \frac{n}{n-1}$, and the corresponding minimization problem consists in establishing the existence of minimizers for the least energy levels:
\[
c = \inf_{u \in X} \Phi(u)\ \ \text{and}\ \ d = \inf_{u \in Y} \Phi(u).
\]
Some subcritical cases have been treated in \cite{KS}, while a few critical ones have been object of study for example in \cite{BW, BS, C1, C2, D, LS}.

The minimization of $\Phi$ in the sets $X$ and $Y$ (with $r = 1$) is also motivated by the existence of extremal functions (i.e. nonzero functions that attain equality) for some classical functional inequalities such as Poincaré, Poincaré-Wirtinger, $L^q$ Poincaré-Sobolev and $L^q$ Poincaré-Wirtinger-Sobolev inequalities for $1 \leq q \leq \frac{n}{n-1}$. More specifically, their respective sharp versions on $BV(\Omega)$ state that
\begin{itemize}
\item[$\blacktriangleright$] Poincaré inequality (${\cal P}$):\\ There exists an optimal constant $\lambda_1 > 0$ such that $\lambda_1\, \Vert u \Vert_{L^1(\Omega)} \leq |Du|(\Omega) + \Vert \tilde{u} \Vert_{L^1(\partial\Omega)}$;

\item[$\blacktriangleright$] Poincaré-Wirtinger inequality (${\cal PW}$):\\ There exists an optimal constant $\mu_1 > 0$ such that $\mu_1\, \Vert u - u_\Omega \Vert_{L^1(\Omega)} \leq |Du|(\Omega)$;

\item[$\blacktriangleright$] Poincaré-Sobolev inequality (${\cal PS}$):\\ There exists an optimal constant $\lambda_q > 0$ such that $\lambda_q\, \Vert u \Vert_{L^q(\Omega)} \leq |Du|(\Omega) + \Vert \tilde{u} \Vert_{L^1(\partial\Omega)}$;

\item[$\blacktriangleright$] Poincaré-Wirtinger-Sobolev inequality (${\cal PWS}$):\\ There exists an optimal constant $\mu_q > 0$ such that $\mu_q\, \Vert u - u_\Omega \Vert_{L^q(\Omega)} \leq |Du|(\Omega)$,
\end{itemize}
where $u_\Omega$ denotes the average of $u$ over $\Omega$.

Some results on existence of extremal functions are well known. For instance, for (${\cal P}$) we refer to \cite{CC, Ch}, for (${\cal PW}$) to \cite{ABM, BS}, for (${\cal PS}$) to \cite{D} and for (${\cal PWS}$) to \cite{ABM, BS, C1, C2}. See also \cite{AGN, MZ, OOR} for refinements in different directions.

Motivated by the existence problem of extremal functions for sharp affine counterparts, some of which weaker than the above inequalities, we develop a theory of minimization for functionals where the term $|Du|(\Omega)$ gives place to the Zhang's affine energy.

In the seminal paper \cite{Z}, Zhang introduced the affine $L^1$ energy (or functional) for functions $u \in W^{1,1}(\R^n)$ given by

\[
{\cal E}_{\R^n}(u) = \alpha_n \left( \int_{\s_{n-1}} \left( \int_{\R^n} | \nabla_\xi u(x) | dx \right)^{-n} d\xi\right)^{-\frac 1n},
\]
where $\alpha_n = (2 \omega_{n-1})^{-1} (n \omega_n)^{1 + 1/n}$. Here, $\nabla_\xi u(x) = \nabla u(x) \cdot \xi$ and $\omega_k$ denotes the volume of the unit ball in $\R^k$. The ``affine" term comes from the invariance property ${\cal E}_{\R^n}(u \circ T) = {\cal E}_{\R^n}(u)$ for every $T \in SL(n)$, where $SL(n)$ denotes the special linear group of $n \times n$ matrices with determinant equal to $1$.

The main result of \cite{Z} ensures that the sharp Sobolev-Zhang inequality

\begin{equation} \label{ASob}
 n \omega_n^{1/n}\, \Vert u \Vert_{L^{\frac{n}{n-1}}(\R^n)} \leq {\cal E}_{\R^n}(u)
\end{equation}
holds for all $u \in W^{1,1}(\R^n)$, with equality attained at characteristic functions of ellipsoids, that is, images of balls under invertible $n \times n$ matrices. Actually, characteristic functions are not in $W^{1,1}(\R^n)$, but rather belong to $BV(\R^n)$.

The Sobolev-Zhang inequality \eqref{ASob} is weaker than the classical sharp $L^1$ Sobolev inequality
\begin{equation} \label{Sob}
 n \omega_n^{1/n}\, \Vert u \Vert_{L^{\frac{n}{n-1}}(\R^n)} \leq \Vert \nabla u \Vert_{L^1(\R^n)},
\end{equation}
since
\begin{equation} \label{comp}
{\cal E}_{\R^n}(u) \leq \Vert \nabla u \Vert_{L^1(\R^n)}
\end{equation}
(see page 194 of \cite{Z}) and also \eqref{comp} is strict for characteristics of ellipsoids other than balls. Zhang also pointed out that the geometric inequality behind \eqref{ASob} is the Petty projection inequality (e.g. \cite{Gardner, S, T}), whereas the geometric inequality behind \eqref{Sob} is the classical isoperimetric inequality.

Later, Wang \cite{Wa} showed that, like in the Sobolev case (e.g. \cite{EG}), the Sobolev-Zhang inequality extends to functions $u \in BV(\R^n)$, where the affine $BV$ energy is expressed naturally by
\[
{\cal E}_{\R^n}(u) = \alpha_n \left( \int_{\s_{n-1}} \left( \int_{\R^n} | \sigma_u(x) \cdot \xi |\, d(|Du|)(x) \right)^{-n} d\xi\right)^{-\frac 1n},
\]
where $\sigma_u: \Omega \rightarrow \R^n$ represents the Radon-Nikodym derivative of $Du$ with respect to its total variation $|Du|$ on $\Omega$, which satisfies $|\sigma_u| = 1$  almost everywhere in $\Omega$ (w.r.t. $|Du|$). Moreover, equality in \eqref{ASob} is achieved precisely by multiples of characteristic functions of ellipsoids, and it also remains weaker than the classical prototype once \eqref{comp} translates to
\begin{equation} \label{comp1}
{\cal E}_{\R^n}(u) \leq |Du|(\R^n).
\end{equation}

Since the Zhang's pioneer work, various improvements and new affine functional inequalities have emerged in a very comprehensive literature. Most of the contributions can be found in the long, but far from complete, list of references \cite{CLYZ, NHJM, HS, HJM1, HJM3, HJM4, HJM5, HJS, KS1, LXZ, LYZ1, LYZ2, LYZ3, Lv, N, N1, N2, Wa, Z}.

Given a function $u \in BV(\Omega)$, denote by $\bar{u}$ its zero extension outside of $\Omega$. The Lipschitz regularity of $\partial \Omega$ guarantees that $\bar{u} \in BV(\R^n)$,

\begin{equation} \label{ebv}
|D \bar{u}|(\R^n) = |Du|(\Omega) + \Vert \tilde{u} \Vert_{L^1(\partial\Omega)}
\end{equation}
and $d(D \bar{u}) = \tilde{u}\, \nu\, d{\cal H}^{n-1}\ {\cal H}^{n-1}$-almost everywhere on $\partial \Omega$, where $\nu$ denotes the unit outward normal to $\partial \Omega$ (see e.g. page 38 of \cite{G}). On the other hand, the latter relation implies that

\[
{\cal E}_{\R^n}(\bar{u}) = \alpha_n \left( \int_{\s_{n-1}} \left( \int_\Omega | \sigma_u(x) \cdot \xi |\, d(|Du|)(x) + \int_{\partial \Omega} |\tilde{u}(x)|\, |\nu(x) \cdot \xi|\, d{\cal H}^{n-1}(x) \right)^{-n} d\xi\right)^{-\frac 1n},
\]
This formula along with \eqref{comp1}, \eqref{ebv} and a reverse Minkowski inequality yield the following comparisons between the affine $BV$ energy of zero extended functions and some local terms:

\begin{itemize}
\item[(C1)] ${\cal E}_{\R^n}(\bar{u}) \leq |Du|(\Omega) + \Vert \tilde{u} \Vert_{L^1(\partial\Omega)}$ for all $u \in BV(\Omega)$;
\item[(C2)] ${\cal E}_{\R^n}(\bar{u}) = {\cal E}_\Omega(u)$ for all $u \in BV_0(\Omega)$;
\item[(C3)] ${\cal E}_{\R^n}(\bar{u}) \geq {\cal E}_\Omega(u) + {\cal E}_{\partial \Omega}(\tilde{u})$ for all $u \in BV(\Omega)$ with $\tilde{u} \neq 0$ on $\partial \Omega$ (a.e.) or no restriction provided that $\partial \Omega$ is non-flat in the sense that $\nu(x) \cdot \xi \neq 0$ on $\partial \Omega$ (a.e.) for every $\xi \in \s^{n-1}$,
\end{itemize}
where $BV_0(\Omega)$ denotes the subspace of $BV(\Omega)$ of functions with zero trace on $\partial \Omega$,

\[
{\cal E}_{\Omega}(u) = \alpha_n \left( \int_{\s_{n-1}} \left( \int_{\Omega} | \sigma_u(x) \cdot \xi |\, d(|Du|)(x) \right)^{-n} d\xi\right)^{-\frac 1n}
\]
and

\[
{\cal E}_{\partial \Omega}(\tilde{u}) = \alpha_n \left( \int_{\s_{n-1}} \left( \int_{\partial \Omega} |\tilde{u}(x)|\, |\nu(x) \cdot \xi|\, d{\cal H}^{n-1}(x) \right)^{-n} d\xi\right)^{-\frac 1n}.
\]
Unlike (C1) and (C2), the comparison (C3) is not straightforward (Corollary \ref{C1}). The required geometric condition is clearly satisfied for many domains which include balls. The above expressions are affine invariants in the sense that ${\cal E}_{\Omega}(u \circ T) = {\cal E}_{T(\Omega)}(u)$ and ${\cal E}_{\partial \Omega}(\tilde{u} \circ T) = {\cal E}_{\partial T(\Omega)}(\tilde{u})$ for every $T \in SL(n)$.

It is worth observing from (C1) that the term ${\cal E}_{\R^n}(\bar{u})$ weakens the right-hand side of (${\cal P}$) and (${\cal PS}$) and this fact encourages us to investigate the new functional $\Phi_{\cal A} : BV(\Omega) \rightarrow \R$,
\[
\Phi_{\cal A}(u) = {\cal E}_{\R^n}(\bar{u}) + \int_{\Omega} a |u|\, dx + \int_{\partial \Omega} b |\tilde{u}|\, d{\cal H}^{n-1}.
\]
Evoking the trace embedding and \eqref{comp1}, one sees that $\Phi_{\cal A}$ is well defined for bounded weights $a$ and $b$.

Consider the least energy levels of $\Phi_{\cal A}$ on $X$ and $Y$:
\[
c_{\cal A} = \inf_{u \in X} \Phi_{\cal A}(u)\ \ \text{and}\ \ d_{\cal A} = \inf_{u \in Y} \Phi_{\cal A}(u).
\]
Note that, by \eqref{comp1}, it may happen that $c_{\cal A} = - \infty$ or $d_{\cal A} = - \infty$ depending on the function $b$. However, both levels are finite if one assumes for instance that $b$ is nonnegative.

Minimization problems for $\Phi_{\cal A}$ are affine variants of those similar for $\Phi$, as can be seen by replacing the term ${\cal E}_{\R^n}(\bar{u})$ by $|D \bar{u}|(\R^n)$ and using \eqref{ebv}. On the other hand, inherent to the nature of our variational problems, some intricate points arise in the search for minimizers. We gather the main ones below:
\begin{itemize}
\item[(P1)] The geometry of $\Phi_{\cal A}$ is non-coercive on $BV(\Omega)$;
\item[(P2)] The functional $u \mapsto {\cal E}_{\R^n}(\bar{u})$ is not convex on $BV(\Omega)$;
\item[(P3)] Sequences in $BV(\Omega)$ with bounded affine $BV$ energy can have unbounded total variation;
\item[(P4)] The continuous immersion $BV(\Omega) \hookrightarrow L^{\frac{n}{n-1}}(\Omega)$ is no longer compact.
\end{itemize}
For the points (P1) and (P3) we refer to the example on page 17 of \cite{HJM5}. Already (P2) can indirectly be checked by combining Theorem \ref{T7} and Proposition \ref{P5}.

The absence of an adequate variational structure leads us to concentrate attention on three central ingredients to be discussed respectively in Sections 3, 4 and 5:
\begin{itemize}
\item[(I1)] The functional $u \mapsto {\cal E}_{\R^n}(\bar{u})$ is weakly* lower semicontinuous on $BV(\Omega)$;
\item[(I2)] The set $B_{\mathcal A}(\Omega) = \{u \in BV(\Omega):\ \Vert u \Vert_{L^1(\Omega)} + {\cal E}_{\R^n}(\bar{u}) \leq 1\}$ is compact in $L^q(\Omega)$ for any $1 \leq q < \frac{n}{n-1}$;
\item[(I3)] Minimizing sequences to $\Phi_{\cal A}$ in $X$ are compact in $L^{\frac{n}{n-1}}(\Omega)$, provided that $c_{\cal A} < n \omega_n^{1/n}$. The same conclusion holds true for $Y$ when one assumes $d_{\cal A} < n \omega_n^{1/n}$.
\end{itemize}

Before we go any further, we have a few comments on each above assertion.

Using the solution of the famous $L^1$-Minkowski problem (see e.g. \cite{LYZ3}), Ludwig \cite{L} was able to show that $u \mapsto {\cal E}_{\R^n}(u)$ is strongly continuous on $W^{1,1}(\R^n)$. After regarding the $BV$-Minkowski problem, Wang \cite{Wa} established the continuity on $BV(\R^n)$ with respect to the strict topology. In Section 3, we prove the strong lower semicontinuity on $L^1_{loc}(\R^n)$ for functions with controlled total variation (Theorem \ref{T8}). The argument bases mainly on a characterization of vanishing of ${\cal E}_{\R^n}(u)$ (Theorem \ref{T7}) and the theory of relaxation and weak semicontinuity for integral functionals by Goffman and Serrin \cite{GS}. As a by-product, $u \mapsto {\cal E}_{\R^n}(\bar{u})$ is strictly continuous and weakly* lower semicontinuous on $BV(\Omega)$, which includes (I1) (Corollary \ref{C2}). The affine compactness stated in (I2) (Theorem \ref{T9}), according to (P3), doesn't follow readily from the Rellich-Kondrachov theorem for $BV$ functions. Its proof makes use of Proposition \ref{P3} and a result due to Huang and Li \cite{HL} which gives a boundedness of the total variation of a function in terms of its affine $BV$ energy unless a suitable affine transformation. Lastly, the assertions (I1) and (I2) along with Proposition \ref{P5} play an essential role in the study of (I3).


Throughout the work, we assume the general assumptions already mentioned:
\begin{itemize}
\item[(H)] $a \in L^\infty(\Omega)$, $b \in L^\infty(\partial \Omega)$ and $b \geq 0$ on $\partial \Omega$.
\end{itemize}

The first result considers subcritical affine minimization problems.

\begin{teor} \label{T1}
The levels $c_{\cal A}$ and $d_{\cal A}$ are attained for any $1 \leq q, r < \frac{n}{n-1}$.
\end{teor}

The next one covers critical cases.

\begin{teor} \label{T2}
The levels $c_{\cal A}$ and $d_{\cal A}$ are attained for $q = \frac{n}{n-1}$ and any $1 \leq r \leq \frac{n}{n-1}$, provided that $0 < c_{\cal A} < n \omega_n^{1/n}$ and $0 < d_{\cal A} < n \omega_n^{1/n}$, respectively.
\end{teor}

The reasoning used in the proof of Theorems \ref{T1} and \ref{T2} produces analogous statements on the space $BV_0(\Omega)$, thanks to its weak* closure in $BV(\Omega)$ (Proposition \ref{P2}). More precisely, using (C2), the functional $\Phi_{\cal A}$, when computed at functions with zero trace in $BV(\Omega)$, becomes
\[
\Phi_{\cal A}(u) = {\cal E}_{\Omega}(u) + \int_{\Omega} a |u|\, dx.
\]
Denote by $c_{{\cal A},0}$ and $d_{{\cal A},0}$ the respective least energy levels of $\Phi_{\cal A}$ on the sets $X_0 = X \cap BV_0(\Omega)$ and $Y_0 = Y \cap BV_0(\Omega)$.

\begin{teor} \label{T3}
The levels $c_{{\cal A},0}$ and $d_{{\cal A},0}$ are attained for any $1 \leq q, r < \frac{n}{n-1}$.
\end{teor}

\begin{teor} \label{T4}
The levels $c_{{\cal A},0}$ and $d_{{\cal A},0}$ are attained for $q = \frac{n}{n-1}$ and any $1 \leq r \leq \frac{n}{n-1}$, provided that $0 < c_{{\cal A},0} < n \omega_n^{1/n}$ and $0 < d_{{\cal A},0} < n \omega_n^{1/n}$, respectively.
\end{teor}

The Sobolev-Zhang inequality on $BV(\R^n)$ yields the sharp affine variants of (${\cal P}$) and (${\cal PW}$) and also of (${\cal PS}$) and (${\cal PWS}$) for $1 \leq q \leq \frac{n}{n-1}$:

\begin{itemize}
\item[$\blacktriangleright$] Affine Poincaré inequality (${\cal AP}$):\\ There exists an optimal constant $\lambda^{\cal A}_1 > 0$ such that $\lambda^{\cal A}_1\, \Vert u\Vert_{L^1(\Omega)} \leq {\cal E}_{\R^n}(\bar{u})$;

\item[$\blacktriangleright$] Affine Poincaré-Wirtinger inequality (${\cal APW}$):\\ There exists an optimal constant $\mu^{\cal A}_1 > 0$ such that $\mu^{\cal A}_1\, \Vert u - u_\Omega\Vert_{L^1(\Omega)} \leq {\cal E}_{\R^n}(\bar{u})$;

\item[$\blacktriangleright$] Affine Poincaré-Sobolev inequality (${\cal APS}$):\\ There exists an optimal constant $\lambda^{\cal A}_q > 0$ such that $\lambda^{\cal A}_q\, \Vert u\Vert_{L^q(\Omega)} \leq {\cal E}_{\R^n}(\bar{u})$;

\item[$\blacktriangleright$] Affine Poincaré-Wirtinger-Sobolev inequality (${\cal APWS}$):\\ There exists an optimal constant $\mu^{\cal A}_q > 0$ such that $\mu^{\cal A}_q\, \Vert u - u_\Omega\Vert_{L^q(\Omega)} \leq {\cal E}_{\R^n}(\bar{u})$.

\end{itemize}

As remarked before, (${\cal AP}$) and (${\cal APS}$) are weakened versions of (${\cal P}$) and (${\cal PS}$), respectively. As for (${\cal APW}$) and (${\cal APWS}$), perhaps one would expect that the term ${\cal E}_{\R^n}(\bar{u})$ could be replaced by ${\cal E}_{\Omega}(u)$ in view of (${\cal PW}$) and (${\cal PWS}$). However, such exchange is not possible, once there are non-constant functions in $BV(\Omega)$ with zero affine $BV$ energy on $\Omega$ (Theorem \ref{T7}). In other words, there exist no positive constant $A$ so that $A\, \Vert u - u_\Omega\Vert_{L^q(\Omega)} \leq {\cal E}_{\Omega}(u)$ is valid for all $u \in BV(\Omega)$. This is a sore point of theory that contrasts drastically with the classical case.

It also deserves to be noticed that ${\cal E}_{\R^n}(\bar{u})$ and $|D u|(\Omega)$ are incomparable via a one-way inequality in $BV(\Omega)$. In effect, since ${\cal E}_{\R^n}(\bar{\chi}_\Omega) = {\cal E}_{\partial \Omega}(1) > 0$ and $|D \chi_\Omega |(\Omega) = 0$, there is no constant $C > 0$ such that ${\cal E}_{\R^n}(\bar{u}) \leq C |D u|(\Omega)$ holds for all $u \in BV(\Omega)$. On the other hand, a reverse inequality also fails in view of the example of \cite{HJM5} in $BV_0(\Omega)$. Accordingly, (${\cal APW}$) and (${\cal APWS}$) seem to be natural affine counterparts of (${\cal PW}$) and (${\cal PWS}$), respectively.

Nonetheless, the term ${\cal E}_{\Omega}(u)$ appears on the right-hand side when we restrict ourselves to functions in $BV_0(\Omega)$. In this space, we denote the respective inequalities by (${\cal AP}_0$), (${\cal APW}_0$), (${\cal APS}_0$) and (${\cal APWS}_0$).

A direct application of Theorems \ref{T1} and \ref{T3} for $r = 1$ is as follows:

\begin{teor} \label{T5}
The inequalities (${\cal AP}$) and (${\cal APW}$) and also (${\cal APS}$) and (${\cal APWS}$) with $1 \leq q < \frac{n}{n-1}$ admit extremal functions in $BV(\Omega)$. The same conclusion holds true in $BV_0(\Omega)$ for (${\cal AP}_0$), (${\cal APW}_0$), (${\cal APS}_0$) and (${\cal APWS}_0$).
\end{teor}

The study of extremal functions for local affine $L^p$-Sobolev type inequalities has been carried out more recently and, as far as we know, it has only been addressed in the papers \cite{ST} and \cite{HJM5} for functions with zero trace. More specifically, the first of them furnishes extremals for the affine $L^2$-Sobolev inequality on $W^{1,2}_0(\Omega)$, whereas the second one for the affine $L^p$-Poincaré inequality on $W^{1,p}_0(\Omega)$ for any $p > 1$ and on $BV_0(\Omega)$ for $p = 1$. In particular, in \cite{HJM5}, the authors provide an alternative proof of Theorem \ref{T5} for (${\cal AP}_0$) through an elegant approach based on their Lemma 1 and Theorem 9.

In the critical case $q = \frac{n}{n-1}$, one knows from \eqref{ASob} that characteristic functions of ellipsoids in $\Omega$ are extremals of (${\cal APS}$), however, exist no extremal for (${\cal APS}_0$). The usual argument of nonexistence consists in showing, by means of a standard rescaling, that the optimal constant corresponding to $BV_0(\Omega)$ is also $n \omega_n^{1/n}$. The key points are the strict continuity of $u \mapsto {\cal E}_{\R^n}(u)$ on $BV(\R^n)$ (Theorem 4.4 of \cite{Wa}) and the density of $BV^\infty_c(\R^n)$ in $BV(\R^n)$ (Corollary 3.2 of \cite{PT}), where $BV^\infty_c(\R^n)$ denotes the space of bounded functions in $BV(\R^n)$ with compact support.

We close the introduction with an application of Theorems \ref{T1} and \ref{T3} for $r > 1$.

We point out that (${\cal APW}$), (${\cal APWS}$), (${\cal APW}_0$) and (${\cal APWS}_0$) are prototypes of more general affine inequalities depending on $q$ and $r$. Precisely, for each $r \geq 1$, let $m_r: L^r(\Omega) \rightarrow \R$ be the unique function that satisfies
\[
\int_\Omega |u - m_r(u)|^{r-1} (u - m_r(u))\, dx = 0
\]
for all $u \in L^r(\Omega)$. It is important to note that $m_r$ is continuous, $1$-homogeneous and bounded on bounded subsets of $L^r(\Omega)$. Of course, $m_1(u) = u_\Omega$ for $r = 1$. The construction of $m_r$ is canonical and makes use of basic results as the mean value theorem and dominated and monotone convergence theorems.

The properties satisfied by $m_r$ together with \eqref{ASob} produce two new affine inequalities for $1 \leq q, r \leq \frac{n}{n-1}$ that extend (${\cal APW}$), (${\cal APWS}$), (${\cal APW}_0$) and (${\cal APWS}_0$).

\begin{itemize}
\item[$\blacktriangleright$] Generalized affine Poincaré-Wirtinger-Sobolev inequality (${\cal GAPWS}$) on $BV(\Omega)$:\\ There exists an optimal constant $\mu^{\cal A}_{q,r} > 0$ such that $\mu^{\cal A}_{q,r}\, \Vert u - m_r(u)\Vert_{L^q(\Omega)} \leq {\cal E}_{\R^n}(\bar{u})$.

\item[$\blacktriangleright$] Generalized affine Poincaré-Wirtinger-Sobolev inequality (${\cal GAPWS}_0$) on $BV_0(\Omega)$:\\ There exists an optimal constant $\mu^{{\cal A},0}_{q,r} > 0$ such that $\mu^{{\cal A},0}_{q,r}\, \Vert u - m_r(u)\Vert_{L^q(\Omega)} \leq {\cal E}_{\Omega}(u)$.
\end{itemize}

\begin{teor} \label{T6}
The inequalities (${\cal GAPWS}$) and (${\cal GAPWS}_0$) admit extremal functions for any $1 \leq q, r < \frac{n}{n-1}$ respectively in $BV(\Omega)$ and $BV_0(\Omega)$.
\end{teor}

\section{Background on the space $BV(\Omega)$}

This section is devoted to some basic definitions and classical results related to functions of bounded variation. For some complete references on the subject, we refer to the books \cite{ABM}, \cite{EG} and \cite{G}.

Let $\Omega$ be a open subset of $\R^n$ with $n \geq 2$. A function $u \in L^1(\Omega)$ is said to be of bounded variation in $\Omega$, if the distributional derivative of $u$ is a vector-valued Radon measure $Du = (D_1 u, \ldots, D_n u)$ in $\Omega$, that is, $D_i u$ is a Radon measure satisfying
\[
\int_\Omega \varphi D_i u = - \int_\Omega \frac{\partial \varphi}{\partial x_i} u\; dx
\]
for every $u \in C^\infty_0(\Omega)$. The vector space of all functions of bounded variation in $\Omega$ is denoted by $BV(\Omega)$.

The total variation of $u$ is defined by
\begin{eqnarray*}
|Du|(\Omega) &=& \sup\left\{\sum_{i = 1}^n \int_\Omega \psi_i D_i u\; dx:\, \psi = (\psi_1, \ldots, \psi_n) \in C^\infty_0(\Omega, \R^n),\, |\psi| \leq 1 \right\}\\
&=& \sup\left\{- \int_\Omega u\, {\rm div}\, \psi\; dx:\, \psi \in C^\infty_0(\Omega, \R^n),\, |\psi| \leq 1 \right\},
\end{eqnarray*}
where $|\psi| = (\psi_1^2 + \cdots + \psi_n^2)^{1/2}$. The variation $|Du|$ is a positive Radon measure on $\Omega$. Denote by $\sigma_u$ the Radon-Nikodym derivative of $D u$ with respect to $|Du|$. Then, $\sigma_u: \Omega \rightarrow \R^n$ is a measurable field satisfying $|\sigma_u| = 1$ almost everywhere in $\Omega$ (w.r.t. $|Du|$) and $d(Du) = \sigma_u d(|Du|)$.

For $u \in BV(\Omega)$, the Lebesgue-Radon-Nikodym decomposition of the measure $Du$ is given by
\[
Du = \nabla u\, {\cal L} + \sigma_u^s\, |D^s u|,
\]
where $\nabla u$ and $D^s u$ denote respectively the (density) absolutely continuous part and the singular part of $Du$ with respect to the $n$-dimensional Lebesgue measure ${\cal L}$ and $\sigma_u^s$ is the Radon-Nikodym derivative of $D^s u$ with respect to its total variation measure $|D^s u|$. In particular,
\[
|Du| = |\nabla u|\, {\cal L} + |D^s u|.
\]

The space $BV(\Omega)$ is Banach with respect to the norm
\[
\Vert u \Vert_{BV(\Omega)} = \Vert u \Vert_{L^1(\Omega)} + |Du|(\Omega),
\]
however it is neither separable nor reflexive.

The strict (intermediate) topology is induced by the metric
\[
d(u,v) = \left| \, |D u|(\Omega) - |D v|(\Omega) \, \right| + \Vert u - v \Vert_{L^1(\Omega)}.
\]

The weak* topology, the weakest of the three ones, is quite appropriate for dealing with minimization problems. A sequence $u_k$ converges weakly* to $u$ in $BV(\Omega)$, if $u_k \rightarrow u$ strongly in $L^1(\Omega)$ and $Du_k \rightharpoonup Du$ weakly* in the measure sense, that is,

\[
\int_\Omega \varphi Du_k \rightarrow \int_\Omega \varphi Du
\]
for every $\varphi \in C^\infty_0(\Omega)$.

Assume that $\Omega$ is a bounded open with Lipschitz boundary. We select below some well-known properties that will be used in the next sections:

\begin{itemize}
\item[$\bullet$] Every bounded sequence in $BV(\Omega)$ admits a weakly* convergent subsequence;
\item[$\bullet$] Every weakly* convergent sequence in $BV(\Omega)$ is bounded;
\item[$\bullet$] $BV(\Omega)$ is embedded continuously into $L^q(\Omega)$ for $1 \leq q \leq \frac{n}{n-1}$ and compactly for $1 \leq q < \frac{n}{n-1}$;
\item[$\bullet$] Each function $u \in BV(\Omega)$ admits a boundary trace $\tilde{u}$ in $L^1(\partial \Omega)$ and the trace operator $u \mapsto \tilde{u}$ is continuous on $BV(\Omega)$ with respect to the strict topology;
\item[$\bullet$] For any function $u \in BV(\Omega)$, its zero extension $\bar{u}$ outside of $\Omega$ belongs to $BV(\R^n)$;
\item[$\bullet$] $\Vert u \Vert'_{BV(\Omega)} = |D u|(\Omega) + \Vert \tilde{u} \Vert_{L^1(\partial \Omega)}$ defines a norm on $BV(\Omega)$ equivalent to the usual norm $\Vert u \Vert_{BV(\Omega)}$;
\item[$\bullet$] $W^{1,1}(\R^n)$ is dense in $BV(\R^n)$ with respect to the strict topology.
\end{itemize}

\section{Lower weak* semicontinuity of ${\cal E}_{\R^n}$}

For an open subset $\Omega \subset \R^n$ and $u \in BV(\Omega)$, consider the affine $BV$ energy

\[
{\cal E}_{\Omega}(u) = \alpha_n \left( \int_{\s_{n-1}} \left( \int_{\Omega} | \sigma_{u}(x) \cdot \xi |\, d(|D u|)(x) \right)^{-n} d\xi\right)^{-\frac 1n}.
\]
We start by giving an answer to the question:

\begin{center}
When is the affine energy ${\cal E}_{\Omega}(u)$ zero?
\end{center}

For each $\xi \in \s^{n-1}$, denote by $\Psi_\xi$ the functional on $BV(\Omega)$,
\[
\Psi_\xi(u) = \int_{\Omega} | \sigma_u(x) \cdot \xi |\, d(|Du|)(x).
\]

\begin{teor} \label{T7}
Let $u \in BV(\Omega)$. Then, ${\cal E}_{\Omega}(u) = 0$ if, and only if, $\Psi_{\tilde{\xi}}(u) = 0$ for some $\tilde{\xi} \in \s^{n-1}$.
\end{teor}

\begin{proof}
The sufficiency is the easy part. In fact, assume that $\Psi_{\xi}(u) > 0$ for all $\xi \in \s^{n-1}$. Thanks to the continuity of $\xi \in \s^{n-1} \mapsto \Psi_{\xi}(u)$, there exists a constant $c > 0$ so that $\Psi_{\xi}(u) \geq c$ for all $\xi \in \s^{n-1}$. But this lower bound immediately yields ${\cal E}_{\Omega}(u) \geq c \alpha_n (n \omega_n)^{-1/n} > 0$.

Conversely, we prove that ${\cal E}_{\Omega}(u) = 0$ whenever $\Psi_{\tilde{\xi}}(u) = 0$ for some $\tilde{\xi} \in \s^{n-1}$. Let $m \in \N$ be the maximum number of linearly independent vectors $\xi \in \s^{n-1}$ such that $\Psi_\xi(u) = 0$. If $m = n$, then clearly $Du = 0$ in $\Omega$ and thus, by \eqref{comp1}, we have ${\cal E}_{\Omega}(u) = 0$. Else, choose an orthonormal basis $\{\xi_1,\ldots,\xi_n\}$ of $\R^n$ so that $\Psi_{\xi_i}(u) = 0$ for $i = n-m+1,\ldots,n$, which correspond to the last $m$ vectors of basis with $0 < m < n$.

For $x \in \Omega$ and $\xi \in \s^{n-1}$, write
\[
\sigma_{u}(x) = \sigma_1(x) \xi_1 + \cdots + \sigma_n(x) \xi_n\ \ {\rm and}\ \ \xi = a_1 \xi_1 + \cdots + a_n \xi_n.
\]
The condition $\Psi_{\xi_i}(u) = 0$ implies that $\sigma_i(x) = 0$ for $i = n-m+1,\ldots,n$. So, the Cauchy-Schwarz inequality gives
\[
| \sigma_{u}(x) \cdot \xi | = | \sigma_1(x) a_1 + \cdots + \sigma_{n-m}(x) a_{n-m} | \leq \left( a_1^2 + \cdots + a_{n-m}^2 \right)^{1/2}.
\]
Set $a(\xi) = (a_1, \ldots, a_{n-m})$ and $a'(\xi) = (a_{n-m+1}, \ldots, a_n)$. Since $0 < m < n$, we get
\begin{eqnarray*}
\int_{\s^{n-1}} \left( \int_{\Omega} | \sigma_{u}(x) \cdot \xi |\, d(|Du|)(x) \right)^{-n} d\xi &\geq& |D u|(\Omega)^{-n} \int_{\s^{n-1}} |a(\xi)|^{-n} d\xi \\
&\geq& |D u|(\Omega)^{-n} \int_{|a(\xi)| \leq \sqrt{3}/2} |a(\xi)|^{-n} d\xi \\
&\geq& \frac{m \omega_m}{2^{m-1}} |D u|(\Omega)^{-n} \int_{|a(\xi)| \leq \sqrt{3}/2} |a(\xi)|^{-n} da(\xi) \\
&=& (n-m) \omega_{n-m} \frac{m \omega_m}{2^{m-1}} |D u|(\Omega)^{-n} \int_0^{\sqrt{3}/2} \rho^{- m -1} d\rho \\
&=& \infty,
\end{eqnarray*}
and hence ${\cal E}_{\Omega}(u) = 0$.
\end{proof}

An interesting application of Theorem \ref{T7}, of independent interest, is

\begin{cor} \label{C1}
Let $\Omega \subset \R^n$ be a bounded open with Lipschitz boundary. Then,
\[
{\cal E}_{\R^n}(\bar{u}) \geq {\cal E}_\Omega(u) + {\cal E}_{\partial \Omega}(\tilde{u})
\]
for all $u \in BV(\Omega)$ with $\tilde{u} \neq 0$ on $\partial \Omega$ (a.e.) or without any restriction in case $\partial \Omega$ is non-flat, where the definitions of ${\cal E}_{\partial \Omega}(\tilde{u})$ and non-flat boundary were given in the comparison (C3) of the introduction.
\end{cor}

\begin{proof}
Firstly, the identity
\[
{\cal E}_{\R^n}(\bar{u}) = \alpha_n \left( \int_{\s_{n-1}} \left( \int_\Omega | \sigma_u(x) \cdot \xi |\, d(|Du|)(x) + \int_{\partial \Omega} |\tilde{u}(x)|\, |\nu(x) \cdot \xi|\, d{\cal H}^{n-1}(x) \right)^{-n} d\xi\right)^{-\frac 1n}
\]
gives ${\cal E}_{\R^n}(\bar{u}) \geq {\cal E}_\Omega(u)$ and ${\cal E}_{\R^n}(\bar{u}) \geq {\cal E}_{\partial \Omega}(\tilde{u})$. Therefore, if ${\cal E}_\Omega(u) = 0$ or ${\cal E}_{\partial \Omega}(\tilde{u}) = 0$, the conclusion follows.

Assume that ${\cal E}_\Omega(u)$ and ${\cal E}_{\partial \Omega}(\tilde{u})$ are nonzero. Set $g(\xi) = \Psi_\xi(u)$ and $\tilde{g}(\xi) = \tilde{\Psi}_\xi(u)$, where
\[
\tilde{\Psi}_\xi(u) = \int_{\partial \Omega} |\tilde{u}(x)|\, |\nu(x) \cdot \xi|\, d{\cal H}^{n-1}(x).
\]
By Theorem \ref{T7}, we have $g(\xi) > 0$ for all $\xi \in \s^{n-1}$. Moreover, the condition ${\cal E}_{\partial \Omega}(\tilde{u}) \neq 0$ and the assumptions assumed in the statement imply that $\tilde{g}(\xi) > 0$ for all $\xi \in \s^{n-1}$. Then, applying the Minkowski inequality for negative parameters, we get
\begin{eqnarray*}
{\cal E}_{\R^n}(\bar{u}) &=& \alpha_n \left( \int_{\s_{n-1}} \left( g(\xi) + \tilde{g}(\xi) \right)^{-n} d\xi\right)^{-\frac 1n} \\
&\geq& \alpha_n \left( \int_{\s_{n-1}} g(\xi)^{-n} d\xi\right)^{-\frac 1n} + \alpha_n \left( \int_{\s_{n-1}} \tilde{g}(\xi)^{-n} d\xi\right)^{-\frac 1n} \\
&=& {\cal E}_\Omega(u) + {\cal E}_{\partial \Omega}(\tilde{u}).
\end{eqnarray*}
\end{proof}

The next step is to establish the lower weak* semicontinuity of ${\cal E}_{\R^n}$ on $L^1_{loc}(\R^n)$ under uniform boundedness of the total variation. The proof makes use, beyond Theorem \ref{T7}, of essential results by Goffman and Serrin (Theorems 2 and 3 of \cite{GS}). We also refer to \cite{AD} and \cite{APR} and references therein for various extensions and improvements of \cite{GS}.

Let $f: \R^n \rightarrow \R$ be a nonnegative convex function with linear growth, that is, $f(w) \leq M(|w| + 1)$ for all $w \in \R^n$, where $M > 0$ is a constant. Define the recession function $f_\infty : \R^n \rightarrow \R$ associated to $f$ by

\[
f_\infty(w) = \limsup_{t \rightarrow \infty} \frac{f(t w)}{t}.
\]
For $u \in BV(\R^n)$, write $Du = \nabla u\, {\cal L} + \sigma_u^s\, |D^s u|$ and let $\Psi: BV(\R^n) \rightarrow \R$ defined by

\[
\Psi(u) = \int_{\R^n} f(\nabla u(x))\, dx + \int_{\R^n} f_\infty(\sigma_u^s(x))\, d(|D^s u|)(x).
\]

\begin{propo}[Goffman-Serrin Theorem] \label{P1}
The functional $\Psi$ is strongly lower semicontinuous on $L^1_{loc}(\R^n)$.
\end{propo}

\begin{teor} \label{T8}
If $u_k \rightarrow u_0$ strongly in $L^1_{loc}(\R^n)$ and $|D u_k|(\R^n)$ is bounded, then
\[
{\cal E}_{\R^n}(u_0) \leq \liminf_{k \rightarrow \infty} {\cal E}_{\R^n}(u_k).
\]
\end{teor}

\begin{proof}
Let $u_k$ be a sequence converging strongly to $u_0$ in $L^1_{loc}(\R^n)$ such that $|D u_k|(\R^n)$ is bounded. If $\Psi_{\tilde{\xi}}(u_0) = 0$ for some $\tilde{\xi} \in \s^{n-1}$, by Theorem \ref{T7}, we have ${\cal E}_{\R^n}(u_0) = 0$ and the conclusion follows trivially.

It then suffices to assume that $\Psi_{\xi}(u_0) > 0$ for all $\xi \in \s^{n-1}$. Set $f^\xi(w)= |w \cdot \xi|$ for any $\xi \in \s^{n-1}$. Since $f^\xi$ is convex, nonnegative, $1$-homogeneous and $f^\xi_\infty = f^\xi$, we have

\begin{eqnarray*}
\Psi_\xi(u) &=& \int_{\R^n} | \sigma_u(x) \cdot \xi |\, d(|Du|)(x)\\
&=& \int_{\R^n} f^\xi(\frac{\nabla u(x)}{|\nabla u(x)|})\, |\nabla u(x)|\, dx + \int_{\R^n} f^\xi(\sigma_u^s(x))\, d(|D^s u|)(x)\\
&=& \int_{\R^n} f^\xi(\nabla u(x))\, dx + \int_{\R^n} f^\xi_\infty(\sigma_u^s(x))\, d(|D^s u|)(x).
\end{eqnarray*}
Hence, by Proposition \ref{P1}, $\Psi_\xi$ is strongly lower semicontinuous on $L_{loc}^1(\R^n)$, and so

\begin{equation} \label{linf}
\int_{\R^n} | \sigma_{u_0}(x) \cdot \xi |\, d(|Du_0|)(x) \leq \liminf_{k \rightarrow \infty} \int_{\R^n} | \sigma_{u_k}(x) \cdot \xi |\, d(|Du_k|)(x).
\end{equation}
We now ensure the existence of a constant $c_0 > 0$ and an integer $k_0 \in \N$, both independent of $\xi \in \s^{n-1}$, such that, for any $k \geq k_0$,  

\begin{equation} \label{loweri}
\int_{\R^n} | \sigma_{u_k}(x) \cdot \xi |\, d(|D u_k|)(x) \geq c_0.
\end{equation}
Otherwise, module a renaming of indexes, we get a sequence $\xi_k \in \s^{n-1}$ such that $\xi_k \rightarrow \tilde{\xi}$ and

\[
\int_{\R^n} | \sigma_{u_k}(x) \cdot \xi_k |\, d(|D u_k|)(x) \leq \frac{1}{k}.
\]
Using the assumption that $|D u_k|(\R^n)$ is bounded, we find a constant $C_1 > 0$ such that

\[
\int_{\R^n} | \sigma_{u_k}(x) \cdot \tilde{\xi} |\, d(|D u_k|)(x) \leq C_1 \Vert \xi_k - \tilde{\xi} \Vert + \frac{1}{k} \rightarrow 0.
\]
Therefore, by \eqref{linf}, we obtain the contradiction $\Psi_{\tilde{\xi}}(u_0) = 0$.

Finally, combining \eqref{linf}, \eqref{loweri} and Fatou's lemma, we derive

\begin{eqnarray*}
&& \int_{\s^{n-1}} \left( \int_{\R^n} | \sigma_{u_0}(x) \cdot \xi |\, d(|Du_0|)(x) \right)^{-n} d\xi \\
&\geq& \int_{\s^{n-1}} \limsup_{k \rightarrow \infty} \left( \int_{\R^n} | \sigma_{u_k}(x) \cdot \xi |\, d(|Du_k|)(x) \right)^{-n} d\xi \\
&\geq& \limsup_{k \rightarrow \infty} \int_{\s^{n-1}} \left( \int_{\R^n} | \sigma_{u_k}(x) \cdot \xi |\, d(|Du_k|)(x) \right)^{-n} d\xi,
\end{eqnarray*}
and thus

\begin{eqnarray*}
{\cal E}_{\R^n}(u_0) &=& \left( \int_{\s^{n-1}} \left( \int_{\R^n} | \sigma_{u_0}(x) \cdot \xi |\, d(|Du_0|)(x) \right)^{-n} d\xi \right)^{-\frac{1}{n}} \\
&\leq& \liminf_{k \rightarrow \infty} \left( \int_{\s^{n-1}} \left( \int_{\R^n} | \sigma_{u_k}(x) \cdot \xi |\, d(|Du_k|)(x) \right)^{-n} d\xi \right)^{-\frac{1}{n}} \\
&=& \liminf_{k \rightarrow \infty} {\cal E}_{\R^n}(u_k).
\end{eqnarray*}
\end{proof}

As an immediate consequence of Theorem \ref{T8}, we have:

\begin{cor} \label{C2}
If $u_k \rightharpoonup u_0$ weakly* in $BV(\Omega)$, then
\[
{\cal E}_{\R^n}(\bar{u}_0) \leq \liminf_{k \rightarrow \infty} {\cal E}_{\R^n}(\bar{u}_k).
\]
\end{cor}

This result is the key point towards the lower weak* semicontinuity of the functional $\Phi_{\cal A}: BV(\Omega) \rightarrow \R$. We recall that

\[
\Phi_{\cal A}(u) = {\cal E}_{\R^n}(\bar{u}) + \int_{\Omega} a |u|\, dx + \int_{\partial \Omega} b |\tilde{u}|\, d{\cal H}^{n-1},
\]
where $a \in L^\infty(\Omega)$ and $b \in L^\infty(\partial \Omega)$ is nonnegative. Since the integral functional on $\Omega$ is clearly weakly* continuous on $BV(\Omega)$, it only remains to discuss the semicontinuity of the boundary integral term.

\begin{propo} \label{P2}
If $u_k \rightharpoonup u_0$ weakly* in $BV(\Omega)$, then

\[
\int_{\partial \Omega} b |\tilde{u}_0|\, d{\cal H}^{n-1} \leq \liminf_{k \rightarrow \infty} \int_{\partial \Omega} b |\tilde{u}_k|\, d{\cal H}^{n-1}.
\]
\end{propo}

\begin{proof}
Let $u_k$ be a sequence converging weakly* to $u_0$ in $BV(\Omega)$. For each $\varepsilon > 0$, we consider the norm $\Vert \cdot \Vert_\varepsilon$ on $BV(\Omega)$,

\[
\Vert u \Vert_\varepsilon = \varepsilon |D u|(\Omega) + \int_{\partial \Omega} (b + \varepsilon) |\tilde{u}|\, d{\cal H}^{n-1}.
\]
Since $b$ is nonnegative, $\Vert \cdot \Vert_\varepsilon$ is equivalent to $\Vert \cdot \Vert_{BV(\Omega)}$ and $\Vert \cdot \Vert'_{BV(\Omega)}$, and so
\begin{equation} \label{pert}
\Vert u_0 \Vert_\varepsilon \leq \liminf_{k \rightarrow \infty} \Vert u_k \Vert_\varepsilon .
\end{equation}
Take a constant $C > 0$ so that $\Vert u_k \Vert'_{BV(\Omega)} \leq C$ and a subsequence $u_{k_j}$ such that
\[
\lim_{j \rightarrow \infty} \int_{\partial \Omega} b |\tilde{u}_{k_j}|\, d{\cal H}^{n-1} = \liminf_{k \rightarrow \infty} \int_{\partial \Omega} b |\tilde{u}_k|\, d{\cal H}^{n-1}.
\]
By \eqref{pert}, for $j$ large, we get
\[
\varepsilon |D u_0|(\Omega) + \int_{\partial \Omega} (b + \varepsilon) |\tilde{u}_0|\, d{\cal H}^{n-1} - \varepsilon \leq C \varepsilon + \int_{\partial \Omega} b |\tilde{u}_{k_j}|\, d{\cal H}^{n-1}.
\]
Letting $j \rightarrow \infty$ and after $\varepsilon \rightarrow 0$, the statement follows as wished.
\end{proof}

Finally, Corollary \ref{C2} and Proposition \ref{P2} lead to

\begin{cor} \label{C3}
The functional $\Phi_{\cal A}$ is lower weakly* semicontinuous on $BV(\Omega)$.
\end{cor}

\section{Subcritical constrained minimizations on $BV(\Omega)$}

We present the proof of Theorems \ref{T1} and \ref{T3}. The main ingredients are Corollary \ref{C3} and the following Rellich-Kondrachov type compactness theorem:

\begin{teor} \label{T9}
The affine ball $B_{\mathcal A}(\Omega)$ is compact in $L^q(\Omega)$ for any $1 \leq q < \frac{n}{n-1}$.
\end{teor}

Its proof demands in turn two preliminary results. The first of them relates weak* convergence of displacements of bounded sequences in $BV(\R^n)$ and strong convergence in $L^q(\R^n)$. Similar results have been established in other spaces, we refer to \cite{AT, AT1, T1, T2} where cocompactness of embeddings are studied in depth. We give the proof for the sake of completeness.

\begin{propo} \label{P3}
Let $u_k$ be a bounded sequence in $BV(\R^n)$. Then, $u_k(\cdot - y_k) \rightharpoonup 0$ locally weakly* in $BV(\R^n)$ for any sequence $y_k$ in $\R^n$ if, and only if, $u_k \rightarrow 0$ strongly in $L^q(\R^n)$ for any $1 < q < \frac{n}{n-1}$.
\end{propo}

\begin{proof}
Assume first that $u_k \rightarrow 0$ strongly in $L^q(\R^n)$ for some $1 < q < \frac{n}{n-1}$. If $v_k = u_k(\cdot - y_k)$ doesn't converge locally weakly* to zero in $BV(\R^n)$ for some sequence $y_k$ in $\R^n$, then there is a bounded open subset $\Omega$ of $\R^n$ and $\varepsilon > 0$ such that, module a subsequence, $\Vert v_k \Vert_{L^1(\Omega)} \geq \varepsilon$ or $|dv_k(\varphi)| \geq \varepsilon$ for some $\varphi \in C^\infty_0(\Omega)$, where $dv(\varphi) = \int_{\Omega} \varphi dv$. Since $v_k$ is bounded in $BV(\R^n)$, one may assume that $v_k \rightharpoonup v$ weakly* in $BV(\Omega)$. Thus, letting $k \rightarrow \infty$ in the two cases, one gets $\Vert v \Vert_{L^1(\Omega)} \geq \varepsilon$ or $|dv(\varphi)| \geq \varepsilon$. On the other hand, one knows that $v_k \rightarrow 0$ strongly in $L^q(\R^n)$ and $v_k \rightarrow v$ strongly in $L^1(\Omega)$, so $v = 0$ in $\Omega$. But this contradicts the last two inequalities.

Conversely, assume that $u_k(\cdot - y_k) \rightharpoonup 0$ locally  weakly* in $BV(\R^n)$ for any sequence $y_k$ in $\R^n$. Choose a fixed $1 < q < \frac{n}{n-1}$ and consider the $n$-cube $Q = (0,1)^n$.

Using the continuity of the Sobolev immersion $BV(Q) \hookrightarrow L^q(Q)$, we deduce that
\begin{eqnarray*}
\int_{Q + y} |u_k|^q\, dx &=& \int_Q |u_k(x - y)|^q\, dx \\
&\leq&  C \Vert u_k(\cdot\, - y) \Vert_{BV(Q)} \left( \int_Q |u_k(x - y)|^q\, dx \right)^{1 - \frac{1}{q}} \\
&=& C \Vert u_k \Vert_{BV(Q + y)} \left( \int_Q |u_k(x - y)|^q\, dx \right)^{1 - \frac{1}{q}}
\end{eqnarray*}
for every $y \in \R^n$, where $C$ is a constant independent of $y$. 

By adding the inequality over $y \in \Z^n$, we obtain
\begin{equation} \label{3}
\int_{\R^n} |u_k|^q\, dx \leq C \Vert u_k \Vert_{BV(\R^n)} \sup_{y \in \Z^n} \left( \int_Q |u_k(x - y)|^q\, dx \right)^{1 - \frac{1}{q}}.
\end{equation}
We claim that the right-hand side of \eqref{3} is finite. Since $u_k$ is bounded in $BV(\R^n)$, it is also bounded in $L^1(\R^n)$ and in $L^{\frac{n}{n-1}}(\R^n)$ by Sobolev inequality. Then, the finiteness follows from the assumption $1 < q < \frac{n}{n-1}$ and a simple interpolation.

Choose $y_k \in \Z^n$ so that
\[
\left( \int_Q |u_k(x - y_k)|^q\, dx \right)^{1 - \frac{1}{q}} \geq \frac{1}{2} \sup_{y \in \Z^n} \left( \int_Q |u_k(x - y)|^q\, dx \right)^{1 - \frac{1}{q}}.
\]
Hence, \eqref{3} gives

\begin{equation} \label{4}
\int_{\R^n} |u_k|^q\, dx \leq 2 C_1 \left( \int_Q |u_k(x - y_k)|^q\, dx \right)^{1 - \frac{1}{q}}
\end{equation}
for some constant $C_1$ independent of $k$.

On the other hand, the strict inequality $q < \frac{n}{n-1}$ allows us to apply the Rellich-Kondrachov compactness theorem to the embedding $BV(Q) \hookrightarrow L^{q}(Q)$ in order to estimate the right-hand side of \eqref{4}. In fact, module a subsequence, we have $v_k = u_k(\cdot - y_k) \rightarrow v$ strongly in $L^q(Q)$. But, by assumption, $v_k \rightharpoonup 0$ locally weakly* in $BV(\R^n)$, and so $v_k \rightarrow 0$ strongly in $L^1(Q)$. Therefore, $v = 0$ in $Q$ and, since $q > 1$, we deduce from \eqref{4} that $u_k \rightarrow 0$ strongly in $L^q(\R^n)$.
\end{proof}

As noted in the introduction, exist no upper bound for $|Du|(\R^n)$ in terms of ${\cal E}_{\R^n} u$ on $BV(\R^n)$. Nonetheless, Huang and Li (Theorem 1.2 of \cite{HL}) proved that such an estimate holds true for functions $u \in W^{1,1}(\R^n)$ unless an adequate affine transformation $T$ depending on $u$. The result is also valid in $BV(\R^n)$, thanks to the necessary tools that were extended by Wang in \cite{Wa}.

\begin{propo}[Huang-Li Theorem] \label{P4}
For any $u \in BV(\R^n)$, one has
\[
d_0 \min_{T \in SL(n)} \left| D(u \circ T) \right|(\R^n) \leq {\cal E}_{\R^n} u,
\]
where $d_0 = 4^{-1} \pi \Gamma(\frac{n+1}2) \Gamma(n + 1)^{\frac 1n} \Gamma(\frac n2 + 1)^{-\frac 1n - 1}$.
\end{propo}

\begin{proof}[Proof of Theorem \ref{T9}]
Let $u_k$ be a sequence in $B_{\mathcal A}(\Omega)$. By Proposition \ref{P4}, there is a matrix $T_k\in SL(n)$ such that $d_0 |D (\bar{u}_k \circ T_k)|(\R^n) \leq {\cal E}_{\R^n} \bar{u}_k$. Note also that $\Vert \bar{u}_k \circ T_k \Vert_{L^1(\R^n)} = \Vert u_k \Vert_{L^1(\Omega)}$, so $v_k = \bar{u}_k \circ T_k$ is bounded in $BV(\R^n)$. We now analyze two possibilities.

Assume first that $\vert T_k\vert\rightarrow\infty$. Let $y_k$ be an arbitrary sequence in $\R^n$. The boundedness of $v_k(\cdot - y_k)$ in $BV(\R^n)$ implies, module a subsequence, that $v_k(\cdot - y_k) \rightharpoonup \bar{v}$ locally weakly* in $BV(\R^n)$. Since $q < \frac{n}{n-1}$, the Rellich-Kondrachov compactness theorem also gives $v_k(\cdot - y_k) \rightarrow \bar{v}$ strongly in $L^q_{loc}(\R^n)$ and $v_k(x - y_k) \rightarrow \bar{v}(x)$ almost everywhere in $\R^n$, up to a subsequence.

Consider the set
\[
X=\lim \inf T_k^{-1}(\Omega + T_k(y_k)) = \bigcup_{m \geq 1} \bigcap_{k \geq m} T_k^{-1}(\Omega + T_k(y_k)).
\]
Since $\vert T_k\vert\rightarrow\infty$ and $\Omega$ is bounded, $X$ has zero Lebesgue measure (e.g page 7 of \cite{ST}). For $x \not\in X$, we have $x\not\in \cap_{k\geq m} T_k^{-1}(\Omega + T_k(y_k))$ for any $m \geq 1$, which yields $T_k(x - y_k) \not\in\Omega$ for every $k$, up to a subsequence. Thus, $\bar{v}(x) = \lim\limits_{k \rightarrow \infty} v_k(x - y_k) = \lim\limits_{k \rightarrow \infty} \bar{u}_k(T_k(x - y_k))= 0$ and hence $v_k(\cdot - y_k) \rightharpoonup 0$ locally weakly* in $BV(\R^n)$ for any sequence $y_k$ in $\R^n$. By Proposition \ref{P3}, $\bar{u}_k \rightarrow 0$ strongly in $L^q(\R^n)$ and so $u_k \rightarrow 0$ strongly in $L^q(\Omega)$.

If $|T_k| \not\rightarrow \infty$, then one may assume that $T_k$ converges to some $T\in SL(n)$. Choose $R > 0$ large enough so that $T^{-1}(\Omega) \subset B_R$ and $T_k^{-1}(\Omega) \subset B_R$ for every $k$. Module a subsequence, we know that $v_k \rightharpoonup v_0$ weakly* in $BV(B_R)$ and $v_k \rightarrow v_0$ strongly in $L^q(B_R)$.

Set $u_0 = v_0 \circ T^{-1}$ in $\Omega$. Notice that $u_0 \in BV(\Omega)$ once $T^{-1}(\Omega) \subset B_R$. Let $\bar{u}_0 \in BV(\R^n)$ be the extension of $u_0$ by zero outside of $\Omega$. Since $T \circ T_k^{-1}$ converges to the identity $I$, by the generalized dominated convergence theorem, it follows that $\Vert \bar{u}_0 \circ T \circ T_k^{-1} - u_0 \Vert_{L^q(\Omega)} \rightarrow 0$. Consequently, since $T_k^{-1}(\Omega) \subset B_R$, we have
\begin{eqnarray*}
\Vert u_k - u_0 \Vert_{L^q(\Omega)} &\leq& \Vert v_k \circ T_k^{-1} - v_0 \circ T_k^{-1} \Vert_{L^q(\Omega)} + \Vert \bar{u}_0 \circ T \circ T_k^{-1} - u_0 \Vert_{L^q(\Omega)}\\
&\leq& \Vert v_k - v_0 \Vert_{L^q(B_R)} + \Vert \bar{u}_0 \circ T \circ T_k^{-1} - u_0 \Vert_{L^q(\Omega)} \rightarrow 0 .
\end{eqnarray*}
\end{proof}

A fact that follows from the proof and deserves to be highlighted is

\begin{cor} \label{C4}
Let $u_k$ be a sequence in $B_{\mathcal A}(\Omega)$ such that $u_k \rightarrow u_0$ strongly in $L^q(\Omega)$ for some $1 \leq q < \frac{n}{n-1}$. If $u_0 \neq 0$, then $u_k$ is bounded in $BV(\Omega)$.
\end{cor}

Theorems \ref{T1} and \ref{T3} can now be proved by using the previous developments.

\begin{proof}[Proof of Theorem \ref{T1}]
Let $u_k$ be a minimizing sequence of $\Phi_{\cal A}$ in $X$. By Hölder's inequality, $u_k$ is bounded in $L^1(\Omega)$ and, since $b \geq 0$ on $\partial \Omega$, the affine energy ${\cal E}_{\R^n} \bar{u}_k$ is also bounded. Therefore, by Theorem \ref{T9}, there exists $u_0 \in BV(\Omega)$ such that $u_k \rightarrow u_0$ strongly in $L^q(\Omega)$. Therefore, $u_0 \in X$ and, by Corollary \ref{C4}, $u_k$ is bounded in $BV(\Omega)$.

Passing to a subsequence, if necessary, one may assume that $u_k \rightharpoonup u_0$ weakly* in $BV(\Omega)$. Then, by Corollary \ref{C3}, we derive
\[
\Phi_{\cal A}(u_0) \leq \liminf_{k \rightarrow \infty} \Phi_{\cal A}(u_k) = c_{{\cal A}},
\]
and thus $u_0$ minimizes $\Phi_{\cal A}$ in $X$.

The same argument also works for a minimizing sequence $u_k$ of $\Phi_{\cal A}$ in $Y$. So, $u_k \rightharpoonup u_0$ weakly* in $BV(\Omega)$ and $u_k \rightarrow u_0$ strongly in $L^q(\Omega)$, module a subsequence, and thus $u_0 \in X$ and
\[
\Phi_{\cal A}(u_0) \leq \liminf_{k \rightarrow \infty} \Phi_{\cal A}(u_k) = d_{{\cal A}}.
\]
It remains to check that $u_0 \in Y$, which it follows readily from Theorem \ref{T9} applied to $L^r(\Omega)$ for $1 \leq r < \frac{n}{n-1}$.
\end{proof}

\begin{proof}[Proof of Theorem \ref{T3}]
Applying Proposition \ref{P2} with $b = 1$, we conclude that the space $BV_0(\Omega)$ is weakly* closed in $BV(\Omega)$. Then, the proof can be performed for the restriction of $\Phi_{\cal A}$ to $BV_0(\Omega)$ exactly as the previous one.
\end{proof}

\section{Critical constrained minimizations on $BV(\Omega)$}

We prove Theorems \ref{T2} and \ref{T4} by using Theorem \ref{T9}, Corollary \ref{C3}, Corollary \ref{C4} and the next result.

Consider the truncation for $h > 0$:
\[
T_h(s) = \min(\max(s,-h), h)\ \text{ and }\ R_h(s) = s - T_h(s).
\]
Proposition 2.3 of \cite{BW} ensures that $|D u|(\R^n) = |D T_h u|(\R^n) + |D R_h u|(\R^n)$ for every $u \in BV(\R^n)$. Unfortunately, such an identity is not valid within the affine setting, however, using Theorem \ref{T7} it is still possible to establish an inequality.

\begin{propo}\label{P5}
For any $u \in BV(\R^n)$,
\[
{\cal E}_{\R^n}(u) \geq {\cal E}_{\R^n}(T_h u) + {\cal E}_{\R^n}(R_h u).
\]
\end{propo}

\begin{proof} We first prove the inequality for functions $u \in W^{1,1}(\R^n)$. From the definition of $T_h(s)$, we have $T_h u, R_h u \in W^{1,1}(\R^n)$ and $\Psi_\xi(u) = \Psi_\xi(T_h u) + \Psi_\xi(R_h u)$ for all $\xi \in \s^{n-1}$, where
\[
\Psi_\xi(u) = \int_{\R^n} | \nabla_\xi u(x) |\, dx.
\]
Note that this decomposition implies ${\cal E}_{\R^n}(u) \geq {\cal E}_{\R^n}(T_h u)$ and ${\cal E}_{\R^n}(u) \geq {\cal E}_{\R^n}(R_h u)$. Thus, the statement follows if ${\cal E}_{\R^n}(T_h u) = 0$ or ${\cal E}_{\R^n}(R_h u) = 0$.

Assuming that ${\cal E}_{\R^n}(T_h u)$ and ${\cal E}_{\R^n}(R_h u)$ are nonzero, by Theorem \ref{T7}, we have $\Psi_\xi(T_h u), \Psi_\xi(R_h u) > 0$ for all $\xi \in \s^{n-1}$. So, by the Minkowski's inequality for negative exponents, we get
\begin{eqnarray*}
{\cal E}_{\R^n}(u) &=& \alpha_n \left( \int_{\s_{n-1}} \left( \Psi_\xi(T_h u) + \Psi_\xi(R_h u) \right)^{-n} d\xi\right)^{-\frac 1n} \\
&\geq& \alpha_n \left( \int_{\s_{n-1}} \left( \Psi_\xi(T_h u) \right)^{-n} d\xi\right)^{-\frac 1n} + \alpha_n \left( \int_{\s_{n-1}} \left( \Psi_\xi(R_h u) \right)^{-n} d\xi\right)^{-\frac 1n} \\
&=& {\cal E}_{\R^n}(T_h u) + {\cal E}_{\R^n}(R_h u).
\end{eqnarray*}
Finally, the inequality extends to $BV(\R^n)$ by using both the density of $W^{1,1}(\R^n)$ in $BV(\R^n)$ and the continuity of $u \in BV(\R^n) \mapsto {\cal E}_{\R^n}(u)$ with respect to the strict topology.
\end{proof}

\begin{proof}[Proof of Theorems \ref{T2} and \ref{T4}] Thanks to the weak* closure of $BV_0(\Omega)$ in $BV(\Omega)$, it is enough to just prove Theorem \ref{T2}.

Let $u_k$ be a minimizing sequence of $\Phi_{\cal A}$ in $X$. Proceeding as in the proof of Theorem \ref{T1}, by Theorem \ref{T9}, we have $u_k \rightarrow u_0$ strongly in $L^1(\Omega)$, module a subsequence. One may also assume that $u_k \rightarrow u_0$ almost everywhere in $\Omega$ and $T_h u_k \rightharpoonup T_h u_0$ weakly in $L^{\frac{n}{n-1}}(\Omega)$.

Using the Sobolev-Zhang inequality on $BV(\R^n)$,
\[
n\omega_n^{1/n} \left( \int_{\R^n} |u|^{\frac{n}{n-1}}\, dx\right)^{\frac{n-1}{n}} \leq {\cal E}_{\R^n}(u),
\]
and that $b$ is nonnegative, we derive
\begin{eqnarray*}
c_{\cal A} &=& \lim_{k \rightarrow \infty}\left( {\cal E}_{\R^n}(\bar{u}_k) + \int_{\Omega} a |u_k|\, dx + \int_{\partial \Omega} b |\tilde{u}_k|\, d{\cal H}^{n-1}\right)\\
&\geq & n\omega_n^{1/n}+\int_{\Omega} a |u_{0}|\, dx,
\end{eqnarray*}
so the condition $c_{\cal A} < n\omega_n^{1/n}$ implies that $u_0 \neq 0$. Hence, by Corollaries \ref{C3} and \ref{C4}, we have $u_k \rightharpoonup u_0$ weakly* in $BV(\Omega)$ and $\Phi_{\cal A}(u_0) \leq c_{\cal A}$. It only remains to show that $u_0 \in X$.

By Proposition \ref{P5}, we easily deduce that
\begin{eqnarray*}
c_{\cal A} &=& \lim_{k \rightarrow \infty} \Phi_{\cal A}(u_k) \\
&\geq & \lim_{k \rightarrow \infty} \left( \Phi_{\cal A}(T_h u_k) + \Phi_{\cal A}(R_h u_k) \right)\\
&\geq & c_{\cal A}\lim_{k \rightarrow \infty}\left( \Vert T_h u_k\Vert_{L^{\frac{n}{n-1}}(\Omega)} + \Vert R_h u_k\Vert_{L^{\frac{n}{n-1}}(\Omega)} \right).
\end{eqnarray*}
Applying Lemma 3.1 of \cite{BW}, we have
\begin{eqnarray*}
c_{\cal A} &\geq & c_{\cal A}\left[ \Vert T_h u_0 \Vert_{\frac{n}{n-1}} + \left( 1 + \Vert R_h u_0 \Vert_{\frac{n}{n-1}}^{\frac{n}{n-1}} - \Vert u_0 \Vert_{\frac{n}{n-1}}^{\frac{n}{n-1}} \right)^{\frac{n-1}{n}} \right].
\end{eqnarray*}
Using the condition $c_{\cal A} > 0$ and letting $h \rightarrow \infty$, one obtains
\[
1\geq \left( \Vert u_0 \Vert_{\frac{n}{n-1}}^{\frac{n}{n-1}} \right)^{\frac{n-1}{n}} + \left( 1 - \Vert u_0 \Vert_{\frac{n}{n-1}}^{\frac{n}{n-1}}\right)^{\frac{n-1}{n}},
\]
and thus $u_0 \in X$ because $u_0 \neq 0$.

If the minimizing sequence $u_k$ of $\Phi_{\cal A}$ is taken in $Y$, the same strategy of proof produces $u_k \rightarrow u_0$ almost everywhere in $\Omega$, $u_0 \in X$ and $\Phi_{\cal A}(u_0) \leq d_{\cal A}$. On the other hand, the first two properties along with Brezis-Lieb Lemma imply that $u_k \rightarrow u_0$ strongly in $L^{\frac{n}{n-1}}(\Omega)$. Finally, since $1 \leq r \leq \frac{n}{n-1}$, it follows that $u_0 \in Y$.
\end{proof}

\n {\bf Acknowledgments:} The first author was partially supported by Fapemig (Universal APQ 00709-18) and the second author was partially supported by CNPq (PQ 302670/2019-0 and Universal 429870/2018-3) and Fapemig (PPM 00561-18).

\end{document}